\documentclass[reqno,12pt]{amsart}
\usepackage{amsmath}
\usepackage{amsthm}
\usepackage{amssymb}

\usepackage[utf8]{inputenc}
\usepackage[T1]{fontenc}
\usepackage{lmodern}
\usepackage{eucal}
\pdfoutput=1 %Needed for microtype and ArXiV.
\usepackage{microtype}
\usepackage{url}
\usepackage{enumitem}
\usepackage{tikz-cd}
\usepackage{bbm}
\usepackage{longtable,booktabs}
\usepackage{faktor}

\usepackage[letterpaper, left = 1in, right = 1in, top = 1in, bottom = 1in]{geometry}

\swapnumbers

\usepackage{hyperref}

\usepackage{xcolor}
\hypersetup{
    colorlinks,
    linkcolor={red!50!black},
    citecolor={blue!50!black},
    urlcolor={blue!80!black}
}

\usepackage[capitalise,nameinlink]{cleveref}
\crefname{subsection}{Subsection}{Subsections}
\crefname{subsubsection}{Subsubsection}{Subsubsections}

\theoremstyle{definition}
\newtheorem{theorem}[subsection]{Theorem}
\newtheorem{construction}[subsection]{Construction}
\newtheorem{defn}[subsection]{Definition}
\newtheorem{example}[subsection]{Example}
\newtheorem{ex}[subsection]{Example}
\newtheorem{cor}[subsection]{Corollary}
\newtheorem{lemma}[subsection]{Lemma}
\newtheorem{prop}[subsection]{Proposition}
\newtheorem{rmk}[subsection]{Remark}
\newtheorem{warning}[subsection]{Warning}
\newtheorem{exercise}[subsection]{Exercise}
\newtheorem{fact}[subsection]{Fact}
\newtheorem{hint}[subsection]{Hint}
\newtheorem{question}[subsection]{Question}
\newtheorem{conjecture}[subsection]{Conjecture}
\newtheorem{obs}[subsection]{Observation}
\newtheorem{problem}[subsection]{Problem}
\newtheorem{goal}[subsection]{Goal}

\newtheorem*{rmk*}{Remark}
\newtheorem*{ex*}{Example}
\newtheorem*{theorem*}{Theorem}
\newtheorem*{defn*}{Definition}
\newcommand{\bfE}{\mathbf{E}}

\newcommand{\bbZ}{\mathbb{Z}}

\newcommand{\bbF}{\mathbb{F}}

\newcommand{\bbE}{\mathbb{E}}

\newcommand{\bbR}{\mathbb{R}}
\newcommand{\bbS}{\mathbb{S}}

\newcommand{\bbA}{\mathbb{A}}

\newcommand{\bbone}{\mathbbm{1}}

\newcommand{\calO}{\mathcal{O}}

\newcommand{\calA}{\mathcal{A}}

\newcommand{\calH}{\mathcal{H}}

\newcommand{\ccat}{\mathcal{C}}
\newcommand{\dcat}{\mathcal{D}}

\newcommand{\Sp}{\mathcal{S}\mathrm{p}}

\newcommand{\Alg}{\mathcal{A}\mathrm{lg}}

\newcommand{\CAlg}{\mathcal{C}\mathrm{Alg}}

\newcommand{\SH}{\mathcal{SH}}

\newcommand{\map}{\operatorname{map}}

\newcommand{\Aut}{\operatorname{Aut}}

\newcommand{\Fib}{\operatorname{Fib}}

\newcommand{\Ext}{\operatorname{Ext}}

\newcommand{\colim}{\operatorname*{colim}}

\newcommand{\Sq}{\mathrm{Sq}}

\newcommand{\res}{\operatorname{res}}
\newcommand{\tr}{\operatorname{tr}}

\newcommand{\Th}{\operatorname{Th}}

\newcommand{\Syn}{\operatorname{Syn}}
\newcommand{\infl}{\operatorname{infl}}

\newcommand{\id}{\operatorname{id}}
\newcommand{\Hol}{\operatorname{Hol}}
\newcommand{\Tot}{\operatorname{Tot}}
\newcommand{\bfmap}{F}%{\operatorname{\mathbf{map}}}
\newcommand{\Bock}{\Upsilon}%{\operatorname{Bock}}
\newcommand{\Be}{\operatorname{Be}}

\newcommand{\integers}{\mathbb{Z}}

\newcommand{\h}{\mathrm{h}}

\newcommand{\cl}{\mathrm{cl}}

\newcommand{\et}{\text{\'et}}

\newcommand{\alg}{\mathrm{alg}}

\newcommand{\bs}{{-}}
 
\newcommand{\ul}{\underline}
\newcommand{\ol}{\overline}

\newcommand\xqed[1]{%
  \leavevmode\unskip\penalty9999 \hbox{}\nobreak\hfill
  \quad\hbox{#1}}
\newcommand\tqed{\xqed{$\triangleleft$}}

\makeatletter
\DeclareRobustCommand{\tvdots}{%
  \vbox{\baselineskip4\p@\lineskiplimit\z@\kern0\p@\hbox{.}\hbox{.}\hbox{.}}}
\makeatother

\makeatletter
\newcommand{\xRightarrow}[2][]{\ext@arrow 0359\Rightarrowfill@{#1}{#2}}
\makeatother

\newcommand{\rightsim}{\xrightarrow{\raisebox{-1pt}{\tiny{$\sim$}}}}

\makeatletter
\@namedef{subjclassname@2020}{\textup{2020} Mathematics Subject Classification}
\makeatother

\usepackage{etoolbox}
\makeatletter
\patchcmd{\@maketitle}{\global\topskip42\p@\relax}
  {\global\topskip42\p@\relax \vspace*{-4.25\baselineskip}}
  {}{}
\makeatother

\usepackage{cite}
\makeatletter
\newcommand{\citecomment}[2][]{\citen{#2}#1\citevar}
\newcommand{\citeone}[1]{\citecomment{#1}}
\newcommand{\citetwo}[2][]{\citecomment[,~#1]{#2}}
\newcommand{\citevar}{\@ifnextchar\bgroup{;~\citeone}{\@ifnextchar[{;~\citetwo}{]}}}
\newcommand{\citefirst}{\@ifnextchar\bgroup{\citeone}{\@ifnextchar[{\citetwo}{]}}}

\makeatother

\usepackage{todonotes}

\begin{document}

\title{Kahn--Priddy theorems via the norm}
\author[William Balderrama]{William Balderrama}
%\date{\today}
\iffalse
\fi
\begin{abstract}
We revisit the Kahn--Priddy theorem from the perspective of modern equivariant homotopy theory. This allows for a short proof that may be applied in other settings with sufficiently robust analogues of multiplicative norms and the Adams isomorphism. We illustrate this by establishing new Kahn--Priddy theorems in $L_n$ and $L_n^f$-local homotopy theory, motivic homotopy theory, and synthetic homotopy theory.
\end{abstract}

\maketitle

\section{Introduction}

\subsection{Background}

The \emph{Kahn--Priddy theorem} \cite{kahnpriddy1972applications} asserts that the transfer
\[
\tr\colon \Sigma^\infty_+ BC_2 \to \bbS
\]
induces a surjection onto the $2$-torsion in the stable stems $\pi_\ast \bbS$. As observed by Adams \cite{adams1973kahn}, Kahn and Priddy's construction proves more: there is a map
\[
\Omega^\infty \bbS \to \Omega^\infty \Sigma^\infty BC_2
\]
of spaces whose composite with $\Omega^\infty\tr$ is a $2$-primary equivalence on connected covers.

Several variations and proofs of this theorem have since appeared. We note in particular an approach sketched by Segal \cite{segal1974operations}, interpreted through the lens of $C_2$-equivariant homotopy theory by Crabb \cite[p.\ 27]{crabb1980z2} and apparently independently by Araki--Iriye \cite[Section 5]{arakiiriye1982equivariant}, who attribute this derivation to H.\ Minami. In equivariant terms, it is derived as a consequence of basic properties of the squaring operation
\[
\Sq\colon \pi_\ast \bbS \to \pi_{\ast(1+\sigma)}^{C_2}\bbS_{C_2}
\]
in the $C_2$-equivariant stable stems. It is by now well understood that this squaring operation exists not just as an operation in the equivariant stable stems, but as a special case of the effect in homotopy groups of the \emph{Hill--Hopkins--Ravenel norm} \cite{hillhopkinsravenel2016nonexistence}
\[
N_e^{C_2}\colon \Sp \to \Sp^{C_2}.
\]
Let us summarize their proof in this language. Fix a prime $p$ and consider the diagram:
\begin{center}\begin{tikzcd}[column sep=huge, row sep=0.8cm]
&\Omega^\infty\Sigma^\infty_+BC_p\ar[dr,"\Omega^\infty\tr"]\ar[d]\\
\Omega^\infty \bbS\ar[r,"N_e^{C_p} - \infl_e^{C_p}"]\ar[ur,dashed]\ar[dr,"0"]&\Omega^\infty \bbS^{C_p}\ar[r,"\res^{C_p}_e"]\ar[d]&\Omega^\infty \bbS\\
&\Omega^\infty \bbS
\end{tikzcd}.\end{center}
The norm and inflation define pointed maps $\Omega^\infty \bbS \to \Omega^\infty \bbS^{C_p}$ from the underlying infinite loop space of the sphere spectrum to that of the fixed points of the $C_p$-equivariant sphere spectrum. These are equalized by geometric fixed points, so their difference lifts to its fiber, and the Adams isomorphism identifies this as a factorization of the composite $\res_e^{C_p}\circ (N_e^{C_p} - \infl_e^{C_p})$ through the transfer. This composite induces $-1$ in positive-dimensional homotopy groups, and so this lift provides a section to the transfer on connected covers.

\pagebreak[0]

Let us also recast the statement of the Kahn--Priddy theorem from the perspective of this proof. Given a category $\ccat$ with finite limits and a grouplike $\bbE_\infty$ ring $R \in \ccat$, define the map
\begin{equation}\label{eq:phip}
\varphi_p\colon R \to R,\qquad \varphi_p(x) = x^p - x.
\end{equation}
Then $\varphi_p$ is a map of pointed objects which is an equivalence after looping once, i.e.\
\[
\Omega \varphi_p\colon \Omega R \rightsim \Omega R
\]
is an equivalence. The above argument shows the following.

\begin{theorem}[``The Kahn--Priddy theorem'']\label{thm:classicalkahnpriddy}
There is a preferred lift in the diagram
\begin{center}\begin{tikzcd}
&\Omega^\infty\Sigma^\infty_+ BC_p\ar[d,"\Omega^\infty\tr"]\\
\Omega^\infty \bbS\ar[r,"\varphi_p"]\ar[ur,dashed]&\Omega^\infty \bbS
\end{tikzcd}\end{center}
of pointed spaces.
\qed
\end{theorem}

Here, by ``preferred lift'' we mean only that a particular lift is constructed rather than abstractly proved to exist. The purpose of this paper is to point out that this proof is sufficiently formal as to hold in any other setting where one has sufficiently robust analogues of the norm and the Adams isomorphism. Rather than attempt to axiomatize this situation, we shall instead illustrate it with examples, establishing new Kahn--Priddy theorems in $L_n$ and $L_n^f$-local homotopy theory, motivic homotopy theory, and synthetic homotopy theory.

\subsection{A smashing Kahn--Priddy theorem}

As observed by Kuhn \cite{kuhn1989morava, clausenmathew2017short}, the theory of the Bousfield--Kuhn functor implies that the transfer
\[
\tr\colon \Sigma^\infty_+ BC_p \to \bbS
\]
admits a \emph{stable} splitting after $T(n)$-localization. This is a purely monochromatic phenomenon: it cannot hold $L_n$ or $L_n^f$-locally (except when $n = 0$). For example, the $L_1$-local transfer for $p=2$ may be identified as the map
\[
L_1\Sigma^\infty_+ BC_2 \simeq L_1 \bbS \oplus \Sigma^{-1}L_1\bbS/(2^\infty) \to L_1 \bbS
\]
which is multiplication by $2$ on the first summand and the boundary on the second summand, and this does not admit a stable splitting. Nonetheless this map still induces a surjection on positive-dimensional homotopy groups. This is not a coincidence.

\begin{theorem}[\ref{thm:lkpg}]\label{introthm:smashingkp}
Let $L$ be a smashing localization of spectra. Then there is a preferred lift in the diagram
\begin{center}\begin{tikzcd}
&\Omega^\infty L \Sigma^\infty_+ BC_p\ar[d,"\Omega^\infty L\tr"]\\
\Omega^\infty L\bbS\ar[r,"\varphi_p"]\ar[ur,dashed]&\Omega^\infty L\bbS
\end{tikzcd}\end{center}
of pointed spaces.
\tqed
\end{theorem}

\begin{cor}[\ref{prop:sigmap}]
Let $L$ be a $p$-local smashing localization of spectra. Then the transfer induces a split surjection
\[
\pi_k L\Sigma^\infty B\Sigma_p \to \pi_k L\bbS
\]
for all $k > 0$.
\tqed
\end{cor}

We actually prove something a little stronger. For any spectrum $Y$, functoriality of the smash product provides a natural pointed map
\[
P\colon \Omega^\infty Y = \map_{\Sp}(\bbS,Y) \to \map_{\Sp}(\bbS^{\otimes p},Y^{\otimes p})\simeq\map_{\Sp}(\bbS,Y^{\otimes p})\simeq \Omega^\infty(Y^{\otimes p}).
\]
On the other hand, if $A$ is a pointed space, then there is a reduced diagonal
\[
\widetilde{\Delta}\colon A \to A^{\times p} \to A^{\wedge p}.
\]
The resulting natural transformation
\[
\varphi_p = P - \Omega^\infty L\Sigma^\infty \widetilde{\Delta} \colon \Omega^\infty L\Sigma^\infty A \to \Omega^\infty L\Sigma^\infty A^{\otimes p}
\]
is seen to specialize to \cref{eq:phip} when $A = S^0$, and we produce a natural lift in the diagram
\begin{equation}\label{eq:naturalkp}
\begin{tikzcd}
&\Omega^\infty L \Sigma^\infty A^{\otimes p}_{\h C_p}\ar[d,"\Omega^\infty L\tr"]\\
\Omega^\infty L\Sigma^\infty A\ar[r,"\varphi_p"]\ar[ur,dashed,"H"]&\Omega^\infty L\Sigma^\infty A^{\otimes p}
\end{tikzcd}.
\end{equation}
This refinement allows us to derive an $L$-local analogue of Kuhn's delooped Kahn--Priddy theorem \cite[Proposition 2.1]{kuhn1982geometry}, see \cref{prop:sigmap}.

\begin{rmk}\label{rmk:klein}
Taking $p=2$ and $L$ to be the identity, the lift in \cref{eq:naturalkp} is a natural transformation
\[
\Omega^\infty \Sigma^\infty A \to \Omega^\infty \Sigma^\infty A^{\otimes 2}_{\h C_2}
\]
which agrees with the model for the \emph{stable Hopf invariant} constructed by Klein in \cite{klein2026on}, and is thus a model for the first attaching map in the Goodwillie tower of the identity. In particular,
\[
P_2^L(A) = \Fib\left(H\colon\Omega^\infty L\Sigma^\infty A \to \Omega^\infty L\Sigma^\infty A^{\otimes 2}_{\h C_2}\right)
\]
\emph{defines} an ``$L$-local metastable approximation to $A$''. It would be interesting to know whether this construction may be characterized by a universal property and to compute examples.
\tqed
\end{rmk}

\subsection{A motivic Kahn--Priddy theorem}

Our next analogue of the Kahn--Priddy theorem is in the setting of motivic homotopy theory. Let $B$ be a qcqs scheme with $p \in \calO_B^\times$, and write
\[
\Sigma^\infty_+ : \calH(B) \leftrightarrows \SH(B): \Omega^\infty
\]
for the adjunction between $B$-motivic spaces and spectra. Write
\[
B_{\et}C_p \in \SH(B)
\]
for the classifying space of \'etale $C_p$-torsors, constructed by Morel and Voevodsky \cite{morelvoevodsky1999a1}. There is a transfer map
\[
\tr\colon \Sigma^\infty_+ B_{\et}C_p \to \bbone_B
\]
of motivic spectra; to be explicit, we take the transfer which arises from the motivic Adams isomorphism of Gepner--Heller \cite{gepnerheller2023tom}. The theory of norms in the motivic homotopy theory of stacks developed by Bachmann \cite{bachmann2022motivic} allows us to interpret the proof of the Kahn--Priddy theorem in the language of motivic homotopy theory, leading to the following.

\begin{samepage}
\begin{theorem}[\ref{thm:motivickp}]\label{introthm:motivickp}
There is a preferred lift in the diagram
\begin{center}\begin{tikzcd}
&\Omega^\infty \Sigma^\infty_+ B_{\et}C_p\ar[d,"\Omega^\infty\tr"]\\
\Omega^\infty\bbone_B\ar[r,"\varphi_p"]\ar[ur,dashed]&\Omega^\infty\bbone_B
\end{tikzcd}\end{center}
of pointed motivic spaces.
\tqed
\end{theorem}
\end{samepage}

\begin{cor}[\ref{cor:motivickphtpy}]
The transfer induces a split surjection 
\[
\ul{\pi}_k \Sigma^\infty_+ B_{\et}C_p \to \ul{\pi}_k \bbone_B
\]
of stable homotopy sheaves for $k\geq 1$.
\qed
\end{cor}

\begin{rmk}
We are not the first to consider the Kahn--Priddy theorem in motivic homotopy theory: in \cite{hornbostel2018some}, Hornbostel proves that an unstable motivic analogue of the Kahn--Priddy theorem for $p=2$ \emph{fails} if one uses the simplicial classifying space $BC_2$ instead of the geometric classifying space $B_{\et}C_2$.
\tqed
\end{rmk}

\subsection{A synthetic Kahn--Priddy theorem}

Our final examples are in the setting of synthetic homotopy theory. Given an Adams-type ring spectrum $E$, write $\Syn_E$ for the category of $E$-synthetic spectra \cite{pstragowski2023synthetic} and
\[
\nu_E\colon \Sp \to \Syn_E
\]
for the synthetic analogue functor. Set
\[
\Sigma^{s,w}\bbone_E =  \Sigma^{s-w}\nu_E(\Sigma^w \bbS),
\]
and write
\[
\tau \in \pi_{0,-1}\bbone_E,\qquad \tau\colon \Sigma \nu_E(\bbS) \to \nu_E(\Sigma\bbS)
\]
for the deformation parameter. The grading is such that, for example,
\[
\pi_{s,w}C(\tau)\cong \Ext_{E_\ast E}^{w-s}(\Sigma^w E_\ast,E_\ast).
\]
As the theory of unstable synthetic homotopy theory has not yet been fully developed, for $X \in \Syn_E$ we will consider the collection of mapping spaces
\[
\Omega^\infty_{w,m} X  = \begin{cases}\map_{\Syn_E}(\Sigma^{0,w}\bbone_E,X \otimes C(\tau^m))&1 \leq m < \infty,\\
\map_{\Syn_E}(\Sigma^{0,w}\bbone_E,X)&m = \infty
\end{cases}
\]
for $1\leq m \leq\infty$ and $0 \leq w < \infty$ as a stand-in for an underlying unstable synthetic object of $X$. In particular $\pi_s \Omega^\infty_{w,m} X = \pi_{s,w}(X\otimes C(\tau^m))$ and $\Omega^\infty_{0,\infty}\Sigma^{0,f}\nu_E(Y) = \Omega^\infty Y$ for $f\geq 0$, and \cref{eq:phip} extends to certain natural maps
\[
\varphi_p\colon \Omega^\infty_{\ast,\ast}\bbone_E \to \Omega^{\infty}_{\ast,\ast}\bbone_E
\]
of pointed spaces which are again equivalences after looping once.

We first specialize to $E = \bbF_2$. As the transfer $\tr\colon \Sigma^\infty_+ BC_2 \to \bbS$ vanishes in mod $2$ homology, its $\bbF_2$-synthetic analogue canonically factors through $\tau$ to provide a synthetic transfer
\[
\tr_{\bbF_2}\colon \Sigma^{0,1}\nu_{\bbF_2}(\Sigma^\infty_+BC_2) \to \bbone_{\bbF_2}.
\]

\begin{samepage}
\begin{theorem}[\ref{thm:newskp}]\label{thm:f2synth}
There is a preferred lift in the diagram
\begin{center}\begin{tikzcd}
&\Omega^\infty_{\ast,\ast} \Sigma^{0,1}\nu_{\bbF_2}(\Sigma^\infty_+ BC_2)\ar[d,"\Omega^\infty_{\ast,\ast}\tr"]\\
\Omega^\infty_{\ast,\ast}\bbone_{\bbF_2}\ar[r,"\varphi_2"]\ar[ur,dashed]&\Omega^\infty_{\ast,\ast} \bbone_{\bbF_2}
\end{tikzcd},\end{center}
which for $(\ast,\ast) = (0,\infty)$ agrees with the lift in \cref{thm:classicalkahnpriddy}.
\tqed
\end{theorem}
\end{samepage}

\begin{cor}[\ref{cor:f2kphtpy}]
The transfer induces compatibly split surjections
\begin{align*}
\bbZ_{(2)}\otimes \pi_{s,w}\nu_{\bbF_2}(\Sigma^\infty BC_2) &\to \bbZ_{(2)}\otimes \pi_{s,w-1}\bbone_{\bbF_2},\\
\pi_{s,w}\nu_{\bbF_2}(\Sigma^\infty BC_2) \otimes C(\tau^m) &\to \pi_{s,w-1}C(\tau^m)
\end{align*}
for all $m,s,w\geq 1$.
\tqed
\end{cor}

\begin{ex}
Specializing to $m=1$, this recovers Lin's \emph{Algebraic Kahn--Priddy theorem} \cite{lin1981algebraic}: the algebraic transfer
\[
\Ext^{f}_{\calA^\vee}(\Sigma^\ast\bbF_2,H_\ast(BC_2,\bbF_2)) \to \Ext_{\calA^\vee}^{f+1}(\Sigma^{\ast+1}\bbF_2,\bbF_2)
\]
is a surjection for $f \geq 0$, where $\calA^\vee$ is the dual Steenrod algebra. Moreover, our work improves on this by constructing a particular section which is compatible with Adams differentials. We give an algebraic description of this section in \cref{prop:sec}; we do not know whether it agrees with the one constructed by Lin. Our proof is completely non-computational to the point where, unlike Singer's conceptual proof of the algebraic Kahn--Priddy theorem \cite{singer1981new}, it does not even require one to have any knowledge of the Steenrod algebra.
\tqed
\end{ex}

\begin{ex}
Write $P^\infty_n$ for the Thom spectrum of $n\sigma$, where $\sigma$ is the tautological line bundle over $BC_2\simeq \mathbb{R}P^\infty$. Following \cref{rmk:klein}, the composite
\[
\pi_{s,w}\bbone_{\bbF_2} \to \pi_{s,w+1}\nu_{\bbF_2}(\Sigma^\infty_+ BC_2) \simeq \pi_{s,w+1}\nu_{\bbF_2}(P^\infty_0) \to \pi_{s,w+1}\nu_{\bbF_2}(P^\infty_n)
\]
of the section to the transfer with the collapse map $P^\infty_0 \to P^\infty_n$ provides an $\bbF_2$-synthetic lift of the stable Hopf invariant
\[
H\colon \Omega^\infty\Sigma^\infty S^n \to \Omega^\infty \Sigma^\infty \Sigma^n P^\infty_n,
\]
an unstable map which a priori need not be compatible with stable Adams filtration at all. Here we mention that, prior to our work, Nikolay Konovalov has shared with us that he has a completely different approach to compatibility of the stable Hopf invariant with stable Adams filtration based on his analysis of the algebraic Goodwillie tower \cite{konovalov2023algebraic} and forthcoming work of Itamar Mor, which moreover applies to higher stages of the Goodwillie tower.
\tqed
\end{ex}

We next specialize to $E = MU$. Under the splitting $\Sigma^\infty_+ BC_p\simeq \bbS \oplus \Sigma^\infty BC_p$, the transfer $\tr\colon \Sigma^\infty_+ BC_p\to \bbS$ may be identified as the degree $p$ map $\bbS \to \bbS$ and a map $\widetilde{\tr}\colon \Sigma^\infty BC_p \to \bbS$ which vanishes in any complex orientable homology theory. In particular, its $MU$-synthetic analogue extends to a synthetic transfer
\[
\tr_{MU}\colon \bbone_{MU} \oplus \Sigma^{0,1}\nu_{MU}(\Sigma^\infty BC_p) \to \bbone_{MU}.
\]

\begin{samepage}
\begin{theorem}[\ref{thm:musynthkp}]\label{thm:musynth}
There is a preferred lift in the diagram
\begin{center}\begin{tikzcd}
&\Omega^\infty_{\ast,\ast}\left( \bbone_{MU}\oplus \Sigma^{0,1}\nu_{MU}(\Sigma^\infty BC_p)\right)\ar[d,"\Omega^\infty_{\ast,\ast}\tr_{MU}"]\\
\Omega^\infty_{\ast,\ast}\bbone_{MU}\ar[r,"\varphi_p"]\ar[ur,dashed]&\Omega^\infty_{\ast,\ast} \bbone_{MU}
\end{tikzcd},\end{center}
which for $(\ast,\ast) = (0,\infty)$ agrees with the lift in \cref{thm:classicalkahnpriddy}.
\tqed
\end{theorem}
\end{samepage}

\begin{cor}[\ref{cor:musynthkp}]\label{cor:bpkp}
The transfer induces compatibly split surjections
\begin{align*}
\bbZ_{(p)}\otimes \pi_{s,w}\nu_{MU}(\Sigma^\infty B\Sigma_p) &\to \bbZ_{(p)}\otimes \pi_{s,w-1}\bbone_{MU},\\
\pi_{s,w}\nu_{MU}(\Sigma^\infty B\Sigma_p) \otimes C(\tau^m) &\to \pi_{s,w-1}C(\tau^m)
\end{align*}
for all $m,s,w\geq 1$.
\tqed
\end{cor}

\begin{ex}
\cref{cor:bpkp} for $m=1$ specializes to an Adams--Novikov analogue of the algebraic Kahn--Priddy theorem. Let us expand this out explicitly. Write
\[
\widetilde{\tr}\colon \Sigma^\infty B\Sigma_p \to \bbS
\]
for the transfer restricted to the summand $\Sigma^\infty B\Sigma_p\subset\Sigma^\infty_+ B\Sigma_p$, and set
\[
\widetilde{B}_{-1} = \Fib\left(\widetilde{\tr}\colon \Sigma^\infty B\Sigma_p \to \bbS\right).
\]
The cofiber sequence
\[
\Sigma^{-1}\bbS \to \widetilde{B}_{-1} \to \Sigma^\infty B\Sigma_p
\]
induces a short exact sequence in $BP$-homology, and so determines an extension
\[
e \in \Ext_{BP_\ast BP}^1(BP_\ast \Sigma^\infty B\Sigma_p,BP_{\ast}\Sigma^{-1}\bbS)
\]
of $BP_\ast BP$-comodules. Yoneda composition with this class defines a map
\[
e\circ (\bs)\colon \Ext_{BP_\ast BP}^f(BP_\ast\Sigma^{w}\bbS,BP_\ast\Sigma^\infty B\Sigma_p)\to\Ext_{BP_\ast BP}^{f+1}(BP_\ast \Sigma^{w}\bbS,BP_\ast\Sigma^{-1} \bbS).
\]
\cref{cor:bpkp} for $m=1$ implies that this is a split surjection for all $f \geq 0$. The $BP_\ast BP$-comodule $BP_\ast B\Sigma_p$ is extremely complicated \cite{davis1978bp}: it is difficult to imagine a direct algebraic proof of this theorem. Moreover, the full strength of our $MU$-synthetic Kahn--Priddy theorem implies that there is even a particular splitting which is compatible with Adams--Novikov differentials, converging to a topological map $\Omega^\infty\bbS \to \Omega^\infty\Sigma^\infty B\Sigma_p$.
\tqed
\end{ex}

Our proof of the synthetic Kahn--Priddy theorem follows the same basic recipe as the others, although the details are somewhat more involved. A rudimentary theory of synthetic norms was implicitly studied in \cite{balderrama2024total}, and forthcoming work of Piessevaux establishes a synthetic Adams isomorphism in the setting of the synthetic equivariant spectra developed by Allen--Piessevaux \cite{allenpiessevaux2025synthetic}. Familiarity with these ingredients was critical to our discovery of the synthetic Kahn--Priddy theorem, although we shall give a direct treatment which is specialized to the particular facts we need.

\subsection*{Acknowledgements}

We thank Lucas Piessevaux for several helpful discussions, and especially for explaining his forthcoming work on synthetic isotropy separation. \cref{lem:fibf2} and \cref{lem:fibmu} in particular are based on these discussions, and we expect more refined statements to appear in his work.

\section{The smashing Kahn--Priddy theorem}\label{sec:localkahnpriddy}

Our goal in this section is to establish the Kahn--Priddy theorem for smashing localizations.

\begin{lemma}\label{lem:normsmash}
Let $L$ be a smashing localization of spectra. Then there is a unique equivalence
\[
N_e^{C_p}(L\bbS)\simeq \infl_e^{C_p}(L\bbS)
\]
in $\CAlg(\Sp^{C_p})$.
\end{lemma}
\begin{proof}
As both $N_e^{C_p}$ and $\infl_e^{C_p}$ are strong monoidal, the condition that $L$ is smashing ensures that both $N_e^{C_p}(L\bbS)$ and $\infl_e^{C_p}(L\bbS)$ are $\otimes$-idempotent in $\Sp^{C_p}$. The zigzag
\[
N_e^{C_p}(L\bbS) \to N_e^{C_p}(L\bbS) \otimes \infl_e^{C_p}(L\bbS) \leftarrow \infl_e^{C_p}(L\bbS)
\]
is seen to be an equivalence on underlying spectra and geometric fixed points, and is thus an equivalence; uniqueness follows from $\otimes$-idempotence.
\end{proof}

Functoriality of the smash product induces for any spectrum $Y$ a pointed map
\[
P\colon \Omega^\infty Y = \map_{\Sp}(\bbS,Y) \to \map_{\Sp}(\bbS^{\otimes p},Y^{\otimes p})\simeq\map_{\Sp}(\bbS,Y^{\otimes p})\simeq \Omega^\infty(Y^{\otimes p}),
\]
and given a pointed space $A$ we set
\[
\varphi_p = P - \Omega^\infty L\Sigma^\infty \widetilde{\Delta} \colon \Omega^\infty L\Sigma^\infty A \to \Omega^\infty L\Sigma^\infty A^{\otimes p},
\]
where $\widetilde{\Delta}\colon A \to A^{\wedge p}$ is the diagonal. The following specializes to \cref{introthm:smashingkp} for $A = S^0$.

\begin{theorem}\label{thm:lkpg}
Let $L$ be a smashing localization of spectra. Then there is a preferred natural lift in the diagram
\begin{center}\begin{tikzcd}
&\Omega^\infty L\Sigma^\infty A^{\otimes p}_{\h C_p}\ar[d,"\Omega^\infty\tr"]\\
\Omega^\infty L \Sigma^\infty A\ar[r,"\varphi_p"]\ar[ur,dashed]&\Omega^\infty L\Sigma^\infty A^{\otimes p}
\end{tikzcd}.\end{center}
\end{theorem}

\begin{proof}
Given a $C_p$-spectrum $X$, abbreviate 
\[
LX = N_e^{C_p}(L\bbS) \otimes X = \infl_e^{C_p}(L\bbS) \otimes X.
\]
For any spectrum $Y$, functoriality of the norm provides a map
\[
N_e^{C_p}\colon \Omega^\infty LY \to \Omega^\infty (LN_e^{C_p}Y)^{C_p}
\]
of pointed spaces whose composite with the geometric fixed point map
\[
\Omega^\infty\Phi^{C_p}\colon  \Omega^\infty (LN_e^{C_p}Y)^{C_p} \to \Omega^\infty \Phi^{C_p}LN_e^{C_p}Y \simeq LY
\]
is the identity, and whose composite with the restriction
\[
\Omega^\infty\res^{C_p}_e\colon \Omega^\infty (LN_e^{C_p}Y)^{C_p} \to \Omega^\infty LY^{\otimes p}
\]
is the transformation $P$ considered above.

On the other hand, the diagonal of a pointed space $A$ refines to a $C_p$-equivariant map $\infl_e^{C_p}(A) \to N_e^{C_p}(A)$. Here, we write again $N_e^{C_p}$ for the unstable norm for pointed spaces, i.e.\ $N_e^{C_p}(A) = A^{\wedge p}$ with cyclic action on coordinates. This suspends and localizes to provide a map
\[
\rho_A\colon L \Sigma^\infty A \to (L\Sigma^\infty N_e^{C_p}A)^{C_p} \simeq (LN_e^{C_p}\Sigma^\infty A)^{C_p}
\]
whose composite with the geometric fixed point map
\[
\Phi^{C_p}\colon (LN_e^{C_p}\Sigma^\infty A)^{C_p} \to \Phi^{C_p}LN_e^{C_p}\Sigma^\infty A \simeq L\Sigma^\infty A
\]
is the identity, and whose composite with the restriction
\[
\res^{C_p}_e\colon (LN_e^{C_p}\Sigma^\infty A)^{C_p} \to L\Sigma^\infty A^{\otimes p}
\]
is $L\Sigma^\infty\widetilde{\Delta}$. Now consider the isotropy separation sequence
\[
EC_{p+} \to \bbS_{C_p} \to \widetilde{E}C_p
\]
in $\Sp^{C_p}$, providing for any $X \in \Sp^{C_p}$ the fiber sequence
\begin{center}\begin{tikzcd}
(E C_{p+} \otimes X)^{C_p}\ar[r]\ar[d,equals]& X^{C_p}\ar[r]\ar[d,equals]&(\widetilde{E}C_p \otimes X)^{C_p}\ar[d,equals]\\
X_{\h C_p}\ar[r]&X^{C_p}\ar[r]&\Phi^{C_p} X
\end{tikzcd}.\end{center}
By the above discussion, we obtain for any pointed space $A$ a natural lift in the diagram
\begin{center}\begin{tikzcd}[column sep=huge, row sep=large]
&\Omega^\infty L\Sigma^\infty A^{\otimes p}_{\h C_p}\ar[d]\ar[dr,"\tr"]\\
\Omega^\infty L\Sigma^\infty A\ar[ur,dashed,"H"]\ar[r,"N_e^{C_p} - \Omega^\infty \rho_A"]\ar[dr,"0"]&\Omega^\infty (LN_e^{C_p} A)^{C_p}\ar[r,"\res^{C_p}_e"]\ar[d,"\Phi^{C_p}"]&\Omega^\infty L\Sigma^\infty A^{\otimes p}\\
&\Omega^\infty L\Sigma^\infty A
\end{tikzcd}\end{center}
for which the horizontal composite is exactly $\varphi_p$.
\end{proof}

This concludes the proof of \cref{introthm:smashingkp}. Several minor variations are possible: for example, the reader may observe that the proof really shows the following.

\begin{prop}
Let $R$ and $C$ be spectra equipped with maps
\[
\nabla\colon N_e^{C_p}(R) \to \infl_e^{C_p}(R),\qquad \Delta\colon \infl_e^{C_p}(C) \to N_e^{C_p}(C)
\]
of $C_p$-spectra which are the identity on geometric fixed points. Then there is a lift in the diagram
\begin{center}\begin{tikzcd}[column sep=7cm, row sep=huge]
&\Omega^\infty(R\otimes C^{\otimes p}_{\h C_p})\ar[d,"\Omega^\infty(R\otimes \tr)"]\\
\Omega^\infty R\otimes C\ar[ur,dashed]\ar[r,"\Omega^\infty(\res^{C_p}_e(\nabla)\otimes C^{\otimes p})\circ P - \Omega^\infty(R \otimes \res^{C_p}_e(\Delta))"]&\Omega^\infty R \otimes C^{\otimes p}
\end{tikzcd}\end{center}
of pointed spaces.
\qed
\end{prop}

For example, this applies when $R$ is a $\bbE_\infty$ ring whose Tate Frobenius and unit maps $R \to R^{t C_p}$ agree and $C = \Sigma^\infty A$ for a pointed space $A$.

We make some further observations. The Kahn--Priddy theorem is traditionally stated $p$-locally, with $\Sigma_p$ in place of $C_p$. One may translate between the two statements as follows. If
\[
\Hol(C_p) = C_p \rtimes \Aut(C_p),
\]
then as the transfer $\Sigma^\infty_+ BC_p \to \bbS$ is $\Aut(C_p)$-equivariant it factors through a map
\[
\tr_p\colon \Sigma^\infty_+ B\Hol(C_p) \to \bbS.
\]
In this way we obtain a commutative diagram
\begin{center}\begin{tikzcd}
\Sigma^\infty_+ B\Sigma_p\ar[r,"\tr"]\ar[dr,"\tr"']&\Sigma^\infty_+BC_p\ar[d,"\tr"]\ar[r,"i"]&\Sigma^\infty_+B\Hol(C_p)\ar[dl,"\tr_p"]\\
&S
\end{tikzcd},\end{center}
where each $\tr$ is a transfer along the relevant subgroup inclusion and $i$ is induced by the inclusion $C_p\subset\Hol(C_p)$. The composite
\[
i\circ\tr\colon \Sigma^\infty_+ B\Sigma_p \to \Sigma^\infty_+ BC_p \to \Sigma^\infty_+ B\Hol(C_p)
\]
is a $p$-local equivalence, and this allows one to translate between theorems about $\tr\colon \Sigma^\infty_+ BC_p \to \bbS$ and $p$-local theorems about $\tr\colon B\Sigma_p \to \bbS$. As a particular case we may derive an $L$-local analogue of Kuhn's delooped Kahn--Priddy theorem \cite[Proposition 2.1]{kuhn1982geometry}.

\begin{prop}\label{prop:sigmap}
Write
\[
\widetilde{\tr}\colon \Sigma^\infty B\Sigma_p \to \bbS
\]
for the transfer restricted to the summand $\Sigma^\infty B\Sigma_p\subset\Sigma^\infty_+ B\Sigma_p$. If $L$ is a $p$-local smashing localization of spectra, then there is a pointed map
\[
\widetilde{H}\colon \Omega^\infty \Sigma L\bbS \to \Omega^\infty \Sigma L\Sigma^\infty B\Sigma_p
\]
for which the composite
\[
\Omega^\infty\widetilde{\tr}\circ \widetilde{H}\colon \Omega^\infty \Sigma L\bbS \to \Omega^\infty \Sigma L \Sigma^\infty B\Sigma_p \to \Omega^\infty \Sigma L\bbS
\]
is an equivalence on simply connected covers.
\end{prop}

\begin{proof}
By the above discussion, we may as well replace $\Sigma_p$ with $\Hol(C_p)$. \cref{thm:lkpg} provides a map $H$ for which the composite
\begin{center}\begin{tikzcd}
\Omega^\infty L\bbS\ar[r,"H"]&\Omega^\infty L\Sigma^\infty_+BC_p\ar[r,"i"]&\Omega^\infty L\Sigma^\infty_+ B\Hol(C_p)\ar[r,"\Omega^\infty L \tr_p"]&\Omega^\infty L\bbS
\end{tikzcd}\end{center}
induces multiplication by $-1$ on $\pi_k$ for $k > 0$. Under the splitting
\[
L\Sigma^\infty_+B\Hol(C_p) \simeq L\bbS \oplus L\Sigma^\infty B\Hol(C_p),
\]
we may identify $L\tr_p$ as a sum of the degree $p$ map on $L\bbS$ and a map $L\widetilde{\tr}_p\colon L\Sigma^\infty B \Hol(C_p) \to L\bbS$. As $L$ is a $p$-local smashing localization, $\pi_k L \bbS$ is $p$-power torsion for $k \neq 0$. Writing
\[
q\colon \Sigma^\infty_+ B\Hol(C_p)\to\Sigma^\infty B\Hol(C_p)
\]
for the projection, it follows that the composite
\begin{center}\begin{tikzcd}[column sep=small]
\Omega^\infty L\bbS\ar[r,"H"]&\Omega^\infty L\Sigma^\infty_+ BC_p \ar[r,"i"]&\Omega^\infty L\Sigma^\infty_+ B\Hol(C_p)\ar[r,"q"]&\Omega^\infty L\Sigma^\infty B\Hol(C_p)\ar[r,"L\tilde{\tr}_p"]&\Omega^\infty L\bbS
\end{tikzcd}\end{center}
still induces an isomorphism on $\pi_k$ for $k > 0$, and so is an equivalence on connected covers. It therefore suffices to prove that the composite $q\circ i \circ H$ deloops as indicated.
If $\ol{\rho}$ is the reduced permutation representation of $\Hol(C_p)$ acting on $C_p$ and $S(\ol{\rho})$ is its unit sphere, then $S(\ol{\rho})_{\h\Hol(C_p)}\to \ast$ is a homology equivalence, and therefore
\[
\Sigma^\infty\Th(\ol{\rho}\downarrow B\Hol(C_p))\simeq \Sigma^\infty B\Hol(C_p).
\]
In other words, we may identify the coassembly map
\[
\Sigma (\bbS^{\otimes p}_{\h \Hol(C_p)}) \to (\Sigma\bbS)^{\otimes p}_{\h \Hol(C_p)}
\]
with the projection
\[
\Sigma q \colon \Sigma \Sigma^\infty_+B\Hol(C_p) \to \Sigma \Sigma^\infty B\Hol(C_p).
\]
\cref{thm:lkpg} for $A = S^1$ therefore provides a map
\[
\widetilde{H}\colon \Omega^\infty \Sigma L\bbS \to \Omega^\infty L(\Sigma \bbS)^{\otimes p}_{\h C_p} \to \Omega^\infty L(\Sigma\bbS)^{\otimes p}_{\h \Hol(C_p)}\simeq \Omega^\infty \Sigma L\Sigma^\infty B \Hol(C_p)
\]
which deloops $q\circ i \circ H$ as needed.
\end{proof}

\section{The motivic Kahn--Priddy theorem}\label{ssec:motivickp}

Let $B$ be a qcqs scheme with $p\in \calO_B^\times$, and write
\[
\Sigma^\infty_+ : \calH(B) \rightleftarrows \SH(B) : \Omega^\infty
\]
for the categories of $B$-motivic spaces and motivic spectra and the usual adjunction between them. In \cite{bachmannelmantoheller2021motivic}, Bachmann--Elmanto--Heller introduce a \emph{motivic extended power} construction. Restricted to $p$th cyclic powers, this is a functor
\[
D^{\mathrm{mot}}_{B_{\et}C_p}\colon \SH(B) \to \SH(B),\qquad Y \mapsto Y^{\otimes p}_{\h\mu_p} = \colim_{B_{\et}C_p}Y^{\otimes p}
\]
defined using a theory of motivic colimits for which, in particular,
\[
(\bbone_B^{\otimes p})_{\h\mu_p}\simeq \Sigma^\infty_+ B_{\et}C_p
\]
is the motivic suspension spectrum of the classifying space of \'etale $C_p$-torsors constructed by Morel and Voevodsky \cite{morelvoevodsky1999a1}.

As explained in \cite[Section 6]{bachmannelmantoheller2021motivic}, this construction may also be derived from the theory of norms in the motivic homotopy theory of stacks developed by Bachmann in \cite{bachmann2022motivic}, extending the theory for schemes developed by Bachmann--Hoyois \cite{bachmannhoyois2021norms}. Write
\[
\Sigma^\infty_+ : \calH^{C_p}(B) \rightleftarrows \SH^{C_p}(B) : \Omega^\infty
\]
for the categories of $C_p$-equivarant $B$-motivic spaces and spectra and the usual adjunction between them, in the sense of Hoyois \cite{hoyois2017six}. Then as a special case of the constructions in \cite[Sections 3.3, 3.4]{bachmann2022motivic}, applied to the tautological maps
\[
i\colon B \to B/\!/C_p,\qquad q\colon B/\!/C_p \to B
\]
of stacks with $C_p$ acting trivially on $B$, one has various functors between motivic categories that can be summarized in the following diagram:
\begin{equation*}\label{eq:context}
\begin{tikzcd}[column sep=1.25cm, row sep=1.25cm]
&\calH(B)\\
\calH(B)\ar[r,"N_e^{C_p}","i_\otimes"']\ar[ur,"(\bs)^{\times p}"]\ar[dr,equals]&\calH^{C_p}(B)\ar[u,"i^\ast","\res^{C_p}_e"']\ar[d,"q_\otimes"',"(\bs)^{C_p}"]&\calH(B)\ar[ul,equals]\ar[l,"\infl^{C_p}_e"',"q^\ast"]\ar[dl,equals]\\
&\calH(B)
\end{tikzcd}
\xRightarrow{\Sigma^\infty_+}
\begin{tikzcd}[column sep=1.25cm, row sep=1.25cm]
&\SH(B)\\
\SH(B)\ar[r,"N_e^{C_p}","i_\otimes"']\ar[ur,"(\bs)^{\otimes p}"]\ar[dr,equals]&\SH^{C_p}(B)\ar[u,"i^\ast","\res^{C_p}_e"']\ar[d,"q_\otimes"',"\Phi^{C_p}"]&\SH(B)\ar[ul,equals]\ar[l,"\infl^{C_p}_e"',"q^\ast"]\ar[dl,equals]\\
&\SH(B)
\end{tikzcd},\end{equation*}
and
\[
Y^{\otimes_p}_{\h\mu_p} = (N_e^{C_p} Y \otimes \bfE C_{p+})/C_p
\]
is the motivic homotopy orbits, in the sense of Gepner--Heller \cite{gepnerheller2023tom}, of the motivic $C_p$-spectrum $N_e^{C_p}Y$

We can now set the stage for the motivic Kahn--Priddy theorem. Write
\[
\tr\colon Y^{\otimes p}_{\h\mu_p} = (N_e^{C_p} Y \otimes \bfE C_{p+})/C_p\simeq (N_e^{C_p}Y\otimes \bfE C_{p+})^{C_p}  \to \res^{C_p}_e(N_e^{C_p}Y\otimes \bfE C_{p+}) \simeq Y^{\otimes p}
\]
for the transfer map obtained from the motivic Adams isomorphism, established by Gepner--Heller \cite{gepnerheller2023tom}. Given a motivic space $A$, write $\Delta_A\colon A \to A^{\times p}$ for the diagonal. As $\Omega^\infty$ is lax monoidal, one has a natural map
\[P : \begin{tikzcd}
\Omega^\infty Y\ar[r,"\Delta_{\Omega^\infty Y}"]&(\Omega^\infty Y)^{\times p}\ar[r]&\Omega^\infty(Y^{\otimes p})
\end{tikzcd}
\]
for any motivic spectrum $Y$. Given a motivic space $A$, set
\[
\varphi_p = P - \Omega^\infty\Sigma^\infty_+ \Delta_A \colon \Omega^\infty\Sigma^\infty_+ A \to \Omega^\infty\Sigma^\infty_+ (A^{\times p})\simeq \Omega^\infty (\Sigma^\infty_+ A)^{\otimes p}.
\]
The following specializes to \cref{introthm:motivickp} for $A = S^0$.

\begin{theorem}\label{thm:motivickp}
There is a preferred natural lift in the diagram
\begin{center}\begin{tikzcd}
&\Omega^\infty(\Sigma^\infty_+ A)^{\otimes p}_{\h\mu_p}\ar[d,"\tr"]\\
\Omega^\infty\Sigma^\infty_+ A\ar[r,"\varphi_p"]\ar[ur,dashed]&\Omega^\infty(\Sigma^\infty_+ A)^{\otimes p}
\end{tikzcd}\end{center}
of motivic spaces.
\end{theorem}
\begin{proof}
With everything in place, the proof is identical to that of the classical Kahn--Priddy theorem as given in detail in \cref{thm:lkpg}.

First, there is a natural mate transformation
\[
N_e^{C_p} \Omega^\infty Y \to \Omega^\infty N_e^{C_p} Y
\]
for $Y \in \SH(B)$. This provides for every motivic spectrum $Y$ a natural transformation
\[
N_e^{C_p}\colon \Omega^\infty Y\simeq (N_e^{C_p}\Omega^\infty Y)^{C_p} \to (\Omega^\infty N_e^{C_p} Y)^{C_p} \simeq \Omega^\infty (N_e^{C_p}Y)^{C_p}
\]
whose composite with the geometric fixed point map
\[
\Omega^\infty\Phi^{C_p}\colon \Omega^\infty (N_e^{C_p} Y)^{C_p} \to \Omega^\infty \Phi^{C_p}N_e^{C_p}E\simeq \Omega^\infty Y
\]
is the identity, and whose composite with the restriction
\[
\Omega^\infty \res^{C_p}_e\colon \Omega^\infty (N_e^{C_p} Y)^{C_p} \to \Omega^\infty Y^{\otimes p}
\]
is the transformation $P$ considered above.

Second, if $A$ is a motivic space, then the unit $A\simeq (N_e^{C_p}A)^{C_p}$ is adjoint to a map 
\[
\rho_A\colon \infl_e^{C_p} A \to N_e^{C_p} A.
\]
Using the mate $\infl_e^{C_p}\Omega^\infty \to \Omega^\infty\infl_e^{C_p}$, we obtain from this a map
\begin{align*}
\Omega^\infty\Sigma^\infty_+\rho_A\colon \Omega^\infty\Sigma^\infty_+A 
&= (\infl_e^{C_p}\Omega^\infty\Sigma^\infty_+ A)^{C_p}
\to (\Omega^\infty \infl_e^{C_p}\Sigma^\infty_+ A)^{C_p}
\\
&\simeq \Omega^\infty (\Sigma^\infty_+ \infl_e^{C_p}A)^{C_p}
\to \Omega^\infty(\Sigma^\infty_+ N_e^{C_p} A)^{C_p}
\simeq \Omega^\infty (N_e^{C_p}\Sigma^\infty_+ A)^{C_p},
\end{align*}
whose composite with the geometric fixed point map
\[
\Omega^\infty\Phi^{C_p}\colon \Omega^\infty (N_e^{C_p}\Sigma^\infty_+  A)^{C_p} \to \Omega^\infty \Phi^{C_p}N_e^{C_p}\Sigma^\infty_+ A \simeq \Omega^\infty \Sigma^\infty_+ A
\]
is the identity, and whose composite with the restriction
\[
\Omega^\infty\res^{C_p}_e\colon \Omega^\infty (N_e^{C_p}\Sigma^\infty_+  A)^{C_p} \to \Omega^\infty (\Sigma^\infty_+ A)^{\otimes p}\simeq \Omega^\infty\Sigma^\infty_+(A^{\times p})
\]
is the diagonal $\Omega^\infty\Sigma^\infty_+\Delta_A$.

Now consider the cofiber sequence
\[
\bfE C_{p+} \to \bbone_{B,C_p} \to \widetilde{\bfE}C_p
\]
in $\SH^{C_p}(B)$, considered in detail in \cite{gepnerheller2023tom}. Given $X \in \SH^{C_p}(B)$, we may identify
\begin{center}\begin{tikzcd}
(\bfE C_{p+} \otimes X)^{C_p}\ar[r]\ar[d,equals]& X^{C_p}\ar[r]\ar[d,equals]&(\widetilde{\bfE}C_p \otimes X)^{C_p}\ar[d,equals]\\
X_{\h\mu_p}\ar[r]&X^{C_p}\ar[r]&\Phi^{C_p} X
\end{tikzcd}.\end{center}
Therefore, by the above discussion, for any motivic space $A$ we obtain a lift in
\begin{center}\begin{tikzcd}[column sep=huge, row sep=huge]
&\Omega^\infty (\Sigma^\infty_+ A)^{\otimes p}_{\h\mu_p}\ar[dr,"\Omega^\infty\tr"]\ar[d]\\
\Omega^\infty\Sigma^\infty_+ A\ar[r,"N_e^{C_p} - \Omega^\infty\Sigma^\infty_+ \rho_A"]\ar[dr,"0"]\ar[ur,dashed]&\Omega^\infty (N_e^{C_p}\Sigma^\infty_+ A)^{C_p}\ar[r,"\Omega^\infty\res^{C_p}_e"]\ar[d,"\Omega^\infty\Phi^{C_p}"]&\Omega^\infty(\Sigma^\infty_+ A)^{\otimes p}\\
&\Omega^\infty\Sigma^\infty_+ A
\end{tikzcd}\end{center}
for which the horizontal composite is exactly $\varphi_p$.
\end{proof}

\begin{cor}\label{cor:motivickphtpy}
The transfer induces a split surjection 
\[
\ul{\pi}_k \Sigma^\infty_+ B_{\et}C_p \to \ul{\pi}_k \bbone_B
\]
of stable homotopy sheaves for $k\geq 1$.
\end{cor}
\begin{proof}
This follows as $\varphi_p\colon \Omega^\infty\bbone_B \to \Omega^\infty\bbone_B$ is an equivalence after looping once.
\end{proof}

\section{The synthetic Kahn--Priddy theorem}

\subsection{Preliminaries}

Given an Adams-type ring spectrum $E$, write $\Syn_E$ for the category of $E$-synthetic spectra \cite{pstragowski2023synthetic} and
\[
\nu_E\colon \Sp \to \Syn_E
\]
for the synthetic analogue functor. Set
\[
\Sigma^{s,w}\bbone_E =  \Sigma^{s-w}\nu_E(\Sigma^w \bbS),
\]
and write
\[
\tau \in \pi_{0,-1}\bbone_E,\qquad \tau\colon \Sigma \nu_E(\bbS) \to \nu_E(\Sigma\bbS)
\]
for the deformation parameter.

\begin{defn}\label{def:syntheticloops}
For $X \in \Syn_E$, we write
\[
\Omega^\infty_{w,m} X  = \begin{cases}\map_{\Syn_E}(\Sigma^{0,w}\bbone_E,X \otimes C(\tau^m))&1 \leq m < \infty,\\
\map_{\Syn_E}(\Sigma^{0,w}\bbone_E,X)&m = \infty
\end{cases}
\]
for $1\leq m \leq \infty$ and $0 \leq w < \infty$.
\tqed
\end{defn}

\begin{rmk}\label{rmk:fibers}
The spaces $\Omega^\infty_{\ast,\ast}(X)$ sit in natural fiber sequences
\[
\Omega^\infty_{w+m,r}X \to \Omega^\infty_{w,m+r}X \to \Omega^\infty_{w,m}X
\]
for all $1 \leq m < \infty$ and $0\leq w < \infty$ and $1 \leq r \leq \infty$.
\tqed
\end{rmk}

\begin{ex}
If $Y$ is any spectrum, then
\[
\Omega^\infty_{0,\infty}\Sigma^{0,f}\nu_E(Y)\simeq\Omega^\infty Y
\]
for all $f\geq 0$ \cite[Theorem 4.58]{pstragowski2023synthetic}.
\tqed
\end{ex}

For any $X\in \Syn_E$, functoriality of smash powers together with the ring structure on $C(\tau^m)$ provides compatible maps
\[
P\colon \Omega_{\ast,\ast}X \to \Omega_{\ast,\ast}X^{\otimes p}.
\]
Explicitly, in bidegree $(w,m)$ this is the composite
\begin{center}\begin{tikzcd}[column sep=-3cm]
\map_{\Syn_E}(\Sigma^{0,w}\bbone_E,X\otimes C(\tau^m))\ar[d,"(\bs)^{\otimes p}"]\\

\map_{\Syn_E}(\Sigma^{0,pw}\bbone_E,X^{\otimes p}\otimes C(\tau^m)^{\otimes p})\ar[d,"{\map_{\Syn_E}(\tau^{(p-1)w},X^{\otimes p}\otimes \mu)}"]\\

\map_{\Syn_E}(\Sigma^{0,w}\bbone_E,X^{\otimes p}\otimes C(\tau^m)),
\end{tikzcd}\end{center}
where for $m = \infty$ one should interpret $C(\tau^m)$ as $\bbone_E$.

\begin{defn}
Given a pointed space $A$ with reduced diagonal $\widetilde{\Delta}_A\colon A \to A^{\wedge p}$, we set
\[
\varphi_p = P - \Omega^\infty_{\ast,\ast}\nu_E(\Sigma^\infty\widetilde{\Delta}_A)\colon \Omega^\infty_{\ast,\ast}\nu_E(\Sigma^\infty A) \to \Omega^\infty_{\ast,\ast}\nu_E(\Sigma^\infty A^{\otimes p}).
\]
Here, we implicitly use the map $\nu_E(\Sigma^\infty A)^{\otimes p} \to \nu_E(\Sigma^\infty A^{\otimes p})$ provided by the lax monoidal structure of $\nu_E$.
\tqed
\end{defn}

\begin{rmk}
Taking $A = S^0$, the map
\[
\varphi_p\colon \Omega^\infty_{\ast,\ast}\bbone_E \to \Omega^\infty_{\ast,\ast}\bbone_E
\]
induces $-1$ on $\pi_k$ for $k > 0$, and is therefore an equivalence after looping once.
\tqed
\end{rmk}

We will not make full use of the category of $E$-synthetic spectra: we will work with explicit cosimplicial resolutions defining the Adams spectral sequence. Given a stable category $\dcat$ and $X,Y\in \dcat$, write
\[
\bfmap_\dcat(X,Y) = \bfmap(X,Y) \in \Sp
\]
for the spectrum of maps from $X$ to $Y$.

\begin{ex}
Taking $\dcat = \Sp^G$ for a group $G$, we have
\[
\bfmap_{\Sp^G}(\bbS_G,X) = X^G
\]
for $X \in \Sp^G$.
\tqed
\end{ex}

\begin{defn}
Given a monoidal stable category $\dcat$, algebra $E \in \Alg(\dcat)$, and $X \in \dcat$, we define
\[
\Bock^E_\star(\Sigma^{s,w}\nu(X)) = \Sigma^{s-w}\Tot\left(\tau_{<\star}\tau_{\geq 0}\bfmap_\dcat(\bbone_\dcat,E^{\otimes\bullet+1}\otimes \Sigma^w X)\right).\tag*{$\triangleleft$}
\]
\end{defn}

The connection to synthetic homotopy is via the following standard fact, compare the discussions in \cite[Section 4.6]{pstragowski2023synthetic} and \cite[Appendix A]{burklundhahnsenger2023boundaries}.

\begin{lemma}\label{lem:ebock}
Suppose that $E$ admits an $\bbA_\infty$ ring structure. For $Y \in \Sp$, we may identify
\[
\Bock^E_\star(\Sigma^{s,w}\nu(Y))\simeq \bfmap_{\Syn_E}(\bbone_E,\Sigma^{s,w}\nu_E(Y) \otimes C(\tau^\star))
\]
for all $s,w\in\integers$. In particular,
\[
\Omega^\infty \Bock_\star^E(\Sigma^{s,w}\nu(Y)) \simeq \Omega^\infty_{0,\star}\Sigma^{s,w}\nu_E(Y).
\]
\end{lemma}
\begin{proof}
As $\Sigma^{s,w}\nu_E(Y) \otimes C(\tau^m)$ is $\nu_E(E)$-nilpotent complete \cite[Lemma A.12]{burklundhahnsenger2023boundaries}, we have
\begin{align*}
\bfmap_{\Syn_E}(\bbone_E,\Sigma^{s,w}\nu_E(Y)\otimes C(\tau^n)) &\simeq \Tot\bfmap_{\Syn_E}(\bbone_E,\nu_E(E)^{\otimes\bullet+1}\otimes \Sigma^{s,w}\nu_E(Y)\otimes C(\tau^m))\\
&\simeq\Sigma^{s-w}\Tot\bfmap_{\Syn_E}(\bbone_E,\nu_E(E^{\otimes\bullet+1}\otimes\Sigma^w Y) \otimes C(\tau^m)),
\end{align*}
the second equivalence by applying \cite[Lemma 4.24]{pstragowski2023synthetic}.
So the lemma follows from the observation that
\[
\bfmap_{\Syn_E}(\bbone_E,\nu_E(Z)\otimes C(\tau^m))\simeq \tau_{<m}\tau_{\geq 0}Z
\]
for any $Z \in \Sp$ which admits an $E$-module structure \cite[Proposition 4.60]{pstragowski2023synthetic}.
\end{proof}

We are interested in the case $\dcat = \Sp^{C_p}$. Fix $E_{C_p} \in \Alg(\Sp^{C_p})$ and suppose that $C_p$ acts trivially on the underlying spectrum $E = \res^{C_p}_e E_{C_p}$. For any $X\in \Sp^{C_p}$, isotropy separation provides a diagram
\begin{center}\begin{tikzcd}[column sep=scriptsize]
E^{\otimes\bullet+1}\otimes X_{\h C_p}\ar[r]\ar[dr,"\tr"]&(E_{C_p}^{\otimes\bullet+1}\otimes X)^{C_p}\ar[d,"\res^{C_p}_e"]\ar[r,"\Phi^{C_p}"]& \Phi^{C_p} E^{\otimes\bullet+1} \otimes \Phi^{C_p} X\ar[r]&E^{\otimes\bullet+1}\otimes \Sigma X_{\h C_p}\\
&E^{\otimes\bullet+1} \otimes \res^{C_p}_e X
\end{tikzcd}\end{center}
of cosimplicial spectra in which the row is a fiber sequence. Applying $\Tot\tau_{<\star}\tau_{\geq 0}(\bs)$, this provides a diagram
\begin{center}\begin{tikzcd}[column sep=scriptsize]
\Bock_\star^E(\nu(X_{\h C_p})) \ar[r]\ar[dr]& \Bock_\star^{E_{C_p}}(\nu(X)) \ar[r,"\Phi^{C_p}"]\ar[d,"\res^{C_p}_e"]& \Bock_\star^{\Phi^{C_p} E}(\nu(\Phi^{C_p}X))\ar[r]&\Bock_\star^E(\Sigma^{1,1}\nu(X_{\h C_p})) \\
&\Bock_\star^E(\nu(X))
\end{tikzcd},\end{center}
only now the row is generally \emph{not} a fiber sequence, as $\tau_{<\star}\tau_{\geq 0}$ does not preserve all fiber sequences. Identifying the fiber of
\[
\Phi^{C_p}\colon \Bock_\star^{E_{C_p}}(\nu(X)) \to \Bock_\star^{\Phi^{C_p} E}(\nu(\Phi^{C_p}X))
\]
in more traditional terms for suitable $E_{C_p}$ and $X$ constitutes the ``Adams isomorphism'' step of the proof of our synthetic Kahn--Priddy theorems, and seems to be the main roadblock to further generalization.

\subsection{The \texorpdfstring{$\bbF_2$}{F\_2}-synthetic Kahn--Priddy theorem}

As the transfer
\[
\tr\colon \Sigma^n P^\infty_n \simeq (\Sigma^n\bbS)^{\otimes 2}_{\h C_2} \to (\Sigma^n \bbS)^{\otimes n} \simeq \Sigma^{2n}\bbS
\]
vanishes in mod $2$ homology, its synthetic analogue canonically factors through $\tau$ to provide a synthetic transfer
\[
\tr_{\bbF_2}\colon \Sigma^{0,1}\nu_{\bbF_2}(\Sigma^n P^\infty_n) \to \Sigma^{2n,2n}\bbone_{\bbF_2}.
\]
Explicitly, the fiber sequence
\begin{center}\begin{tikzcd}
\Sigma^{n-1}\bbS\ar[r]&\Sigma^n P^\infty_{n-1}\ar[r]&\Sigma^nP^\infty_n
\end{tikzcd}\end{center}
obtained from the fiber of the transfer is preserved by $\nu_{\bbF_2}$, and $\tr_{\bbF_2}$ is the associated boundary.

We need a good $C_2$-equivariant lift of $\bbF_2$. The simplest for our purposes is that which is based on the cofree $C_2$-spectrum 
\[
\bbF_2^h = F(EC_{2+},\infl_e^{C_2}\bbF_2)
\]
on $\bbF_2$ with trivial $C_2$-action. The key property that makes $\bbF_2^h$ ``good'' from our perspective is the following.

\begin{lemma}\label{lem:regulareuler}
Let $R$ be a global ring spectrum, and suppose that there exists an invertible class $u_\sigma \in \pi_{1-\sigma}^{C_2}R$. Then multiplication by the Euler class
\[
a_\sigma \colon \pi_\star^{C_2}R \to \pi_{\star-\sigma}^{C_2}R
\]
is an injection on the $RO(C_2)$-graded homotopy groups of the underlying $C_2$-spectrum of $R$.
\end{lemma}
\begin{proof}
By the cofiber sequence
\[
C_{2+} \to \bbS_{C_2} \to \Sigma^\sigma \bbS_{C_2},
\]
multiplication by the Euler class sits in a long exact sequence
\begin{center}\begin{tikzcd}
\cdots\ar[r]&\pi_{\star+1-\sigma}^{C_2}\ar[r,"\res^{C_2}_e"]&\pi_{|\star|}^eR\ar[r,"\tr^{C_2}_e"]&\pi_\star^{C_2}R\ar[r,"a_\sigma"]&\pi_{\star-\sigma}^{C_2}R\ar[r]&\cdots
\end{tikzcd}.\end{center}
To show that $a_\sigma$ is injective, it suffices to show that $\tr^{C_2}_e$ vanishes, for which it suffices to show that $\res^{C_2}_e$ is surjective. Using the invertible class $u_\sigma \in \pi_{1-\sigma}^{C_2}R$, it suffices to show that the restriction
\[
\res^{C_2}_e\colon \pi_\ast^{C_2}R \to \pi_\ast^e R
\]
is a surjection in integer degrees, which holds as the global structure of $R$ provides an inflation
\[
\infl_e^{C_2}\colon \pi_\ast^e R \to \pi_\ast^{C_2}R
\]
satisfying $\res^{C_2}_e\circ \infl_e^{C_2} = \id$.
\end{proof}

Abbreviate
\[
\bbF_2^t = \Phi^{C_2} \bbF_2^h.
\]
The unique ring map
\[
\bbF_2 \to \bbF_2^t
\]
induces an equivalence of Adams spectral sequences, and so we may identify
\[
\Bock_\star^{\bbF_2}(\Sigma^{s,w}\nu(X))\simeq \Bock_\star^{\bbF_2^t}(\Sigma^{s,w}\nu(X))
\]
for any spectrum $X$.

\begin{lemma}\label{lem:fibf2}
The sequence
\begin{center}\begin{tikzcd}
\Bock_\star^{\bbF_2^h}(\nu(\Sigma^{n(1+\sigma)}\bbS_{C_2})) \ar[r,"\Phi^{C_2}"]& \Bock_\star^{\bbF_2}(\nu(\Sigma^n\bbS)) \ar[r]& \Bock_\star^{\bbF_2}(\Sigma^{1,1}\nu(\Sigma^n P^\infty_n))
\end{tikzcd}\end{center}
is a fiber sequence for all $n\in\integers$.
\end{lemma}
\begin{proof}
This sequence is obtained by applying $\Tot \tau_{<\star}\tau_{\geq 0}(\bs)$ to the fiber sequence
\begin{center}\begin{tikzcd}
\left((\bbF_2^h)^{\otimes \bullet+1} \otimes \Sigma^{n(1+\sigma)}\bbS_{C_2}\right)^{C_2}\ar[r,"\Phi^{C_2}"]&( \bbF_2^t)^{\otimes\bullet+1} \otimes \Sigma^n\bbS\ar[r]&\bbF_2^{\otimes\bullet+1}\otimes \Sigma^{n+1}P^\infty_n
\end{tikzcd}.\end{center}
It therefore suffices to show that this fiber sequence induces a short exact sequence on all homotopy groups. As
\[
\Phi^{C_2}\colon \pi_\star (\bbF_2^h)^{\otimes \bullet+1} \to \pi_{\star^{C_2}}(\bbF_2^t)^{\otimes\bullet+1}
\]
is obtained by inverting the Euler class $a_\sigma \in \pi_{-\sigma}\bbS_{C_2}$ of the real sign representation, it suffices to show that this Euler class is regular on the $RO(C_2)$-graded homotopy groups of all tensor powers of $\bbF_2^h$. This follows from \cref{lem:regulareuler}.
\end{proof}

Extending this fiber sequence to the left, we find that the map of towers associated to the synthetic transfer sits in a diagram
\begin{equation}\label{eq:f2adamsiso}\begin{tikzcd}
\Bock_\star^{\bbF_2}(\Sigma^{0,1}\nu(\Sigma^n P^\infty_n)) \ar[r]\ar[dr,"\tr_{\bbF_2}"']& \Bock_\star^{\bbF_2^h}(\nu(\Sigma^{n(1+\sigma)}\bbS_{C_2})) \ar[r,"\Phi^{C_2}"]\ar[d,"\res^{C_2}_e"]& \Bock_\star^{\bbF_2}(\Sigma^n\bbS)\\
&\Bock_\star^{\bbF_2}(\Sigma^{2n}\bbS)
\end{tikzcd}\end{equation}
in which the top row is a fiber sequence. We can now prove the following.

\begin{theorem}\label{thm:newskp}
There is a preferred lift in the diagram
\begin{center}\begin{tikzcd}
&\Omega^\infty_{\ast,\ast} \Sigma^{0,1}\nu(\Sigma^n P^\infty_n)\ar[d,"\Omega^\infty_{\ast,\ast}\tr_{\bbF_2}"]\\
\Omega^\infty_{\ast,\ast}\Sigma^{n,n}\bbone_{\bbF_2}\ar[r,"\varphi_2"]\ar[ur,dashed]&\Omega^\infty_{\ast,\ast} \Sigma^{2n,2n}\bbone_{\bbF_2}
\end{tikzcd}\end{center}
for $n\geq 0$, which for $(\ast,\ast) = (0,\infty)$ agrees with that in \cref{thm:lkpg} for $A=S^n$ and $L$ the identity.
\end{theorem}
\begin{proof}
By the fiber sequences
\[
\Omega^\infty_{m,r} \to \Omega^\infty_{0,m+r} \to \Omega^{\infty}_{0,m}
\]
of \cref{rmk:fibers}, it suffices to consider just $(\ast,\ast) = (0,m)$ for $1 \leq m \leq \infty$. When $m = \infty$ we may simply take the lift in \cref{thm:lkpg} to be our lift here, and so our job is to construct lifts for $m < \infty$ which are compatible with this.

Observe that $\bbF_2^h$, being the cofree $C_2$-spectrum on an $\bbE_\infty$ ring, is a normed $C_2$-ring spectrum. Combined with functoriality of the norm, this provides a map
\[
N_e^{C_2}\colon \Omega^\infty(\bbF_2^{\otimes\bullet+1}\otimes \Sigma^n\bbS) \to \Omega^\infty\left( N_e^{C_2}\bbF_2^{\otimes\bullet+1}\otimes \Sigma^{n(1+\sigma)}\bbS_{C_2}\right)^{C_2} \to \Omega^\infty\left((\bbF_2^h)^{\otimes\bullet+1}\otimes \Sigma^{n(1+\sigma)}\bbS_{C_2}\right)^{C_2}
\]
whose composite with the geometric fixed point map
\[
\Omega^\infty\left((\bbF_2^h)^{\otimes\bullet+1}\otimes \Sigma^{n(1+\sigma)}\bbS_{C_2}\right)^{C_2} \to \Omega^\infty(\bbF_2^t)^{\otimes\bullet+1}\otimes \Sigma^n\bbS
\]
is $\Omega^\infty$ of the map of resolutions associated to the Tate Frobenius $\bbF_2 \to \bbF_2^t$, and whose composite with the restriction
\[
\Omega^\infty\left((\bbF_2^h)^{\otimes\bullet+1}\otimes \Sigma^{n(1+\sigma)}\bbS_{C_2}\right)^{C_2} \to \Omega^\infty\left(\bbF_2^{\otimes\bullet+1}\otimes \Sigma^{2n}\bbS\right)
\]
is the map $P$ induced by functoriality of tensor squares and the product on $\bbF_2$. 

On other hand, the global structure of $\bbF_2^h$ provides an inflation $\infl_e^{C_p}\bbF_2 \to \bbF_2^h$, which combines with the inclusion of fixed points $a_\sigma^n\colon \infl_e^{C_2}(S^n) \to S^{n(1+\sigma)}$ for $n\geq 0$ to provide a map
\[
a_\sigma^n\colon \bbF_2^{\otimes\bullet+1}\otimes \Sigma^n\bbS \to \left((\bbF_2^h)^{\otimes\bullet+1}\otimes \Sigma^{n(1+\sigma)}\bbS_{C_2}\right)^{C_2}
\]
whose composite with the geometric fixed point map
\[
\left((\bbF_2^h)^{\otimes\bullet+1}\otimes \Sigma^{n(1+\sigma)}\bbS_{C_2}\right)^{C_2} \to (\bbF_2^t)^{\otimes\bullet+1}\otimes \Sigma^n\bbS
\]
is the map of resolutions associated to the unit $\bbF_2 \to \bbF_2^t$ and whose composite with the restriction
\[
\left((\bbF_2^h)^{\otimes\bullet+1}\otimes  \Sigma^{n(1+\sigma)}\bbS_{C_2}\right)^{C_2} \to \left(\bbF_2^{\otimes\bullet+1}\otimes \Sigma^{2n}\bbS\right)
\]
is the map induced by the diagonal map $\widetilde{\Delta}\colon S^n \to S^{2n}$.

Applying $\Tot\tau_{<m}(\bs)$ to these maps of cosimplicial spaces provides maps
\[
N_e^{C_2},\Omega^\infty a_\sigma^n\colon \Omega^\infty \Bock_m^{\bbF_2}(\nu(\Sigma^n\bbS))\to \Omega^\infty\Bock_m^{\bbF_2^h}(\nu(\Sigma^{n(1+\sigma)}\bbS_{C_2})),
\]
and we obtain a diagram of the form
\begin{center}\begin{tikzcd}[column sep=huge, row sep=large]
&\Omega^\infty\Bock_m^{\bbF_2}(\Sigma^{0,1}\nu(\Sigma^nP^\infty_n))\ar[dr,"\tr_{\bbF_2}"]\ar[d]\\
\Omega^\infty \Bock_m^{\bbF_2}(\Sigma^n\bbS)\ar[r,"N_e^{C_2} - \Omega^\infty a_\sigma^n"]\ar[dr,"0"]\ar[ur,dashed]&\Omega^\infty\Bock_m^{\bbF_2^h}(\nu(\Sigma^{n(1+\sigma)}\bbS_{C_2}))\ar[r,"\res^{C_2}_e"]\ar[d,"\Phi^{C_2}"]&\Omega^\infty \Bock_m^{\bbF_2}(\Sigma^n\bbS)\\
&\Omega^\infty \Bock_m^{\bbF_2}(\Sigma^{n}\bbS)
\end{tikzcd},\end{center}
where the column is $\Omega^\infty$ of \cref{eq:f2adamsiso} and, under the equivalence $\Omega^\infty \Bock_m^{\bbF_2}(\nu(X)) = \Omega^\infty_{0,m}\nu(X)$, the horizontal composite is the map
\[
\varphi_2\colon \Omega^{\infty}_{0,m}\Sigma^{n,n}\bbone_{\bbF_2} \to \Omega^\infty_{0,m}\Sigma^{2n,2n}\bbone_{\bbF_2}.
\]
As $N_e^{C_2}$ and $\Omega^\infty a_\sigma^n$ are equalized by geometric fixed points by the above discussion, there is a lift in the diagram as indicated, providing the desired factorization of $\varphi_2$ through the synthetic transfer. These lifts are by construction compatible with each other and with the lift of \cref{thm:lkpg}.
\end{proof}

\begin{cor}\label{cor:f2kphtpy}
The transfer induces compatibly split surjections
\begin{align*}
\bbZ_{(2)}\otimes \pi_{s,w}\nu_{\bbF_2}(\Sigma^\infty BC_2) &\to \bbZ_{(2)}\otimes \pi_{s,w-1}\bbone_{\bbF_2},\\
\pi_{s,w}\nu_{\bbF_2}(\Sigma^\infty BC_2) \otimes C(\tau^m) &\to \pi_{s,w-1}C(\tau^m)
\end{align*}
for all $m,s,w\geq 1$.
\end{cor}
\begin{proof}
With $\Sigma^\infty_+BC_2$ in place of $\Sigma^\infty BC_2$, this follows from the $n=0$ case of \cref{thm:newskp}. Under the splitting
\[
\nu_{\bbF_2}(\Sigma^\infty_+ BC_2)\simeq \bbone_{\bbF_2}\oplus\nu_{\bbF_2}(\Sigma^\infty BC_2),
\]
the synthetic transfer
\[
\tr_{\bbF_2}\colon \Sigma^{0,1}\nu_{\bbF_2}(\Sigma^\infty_+ BC_2) \to \bbone_{\bbF_2}
\]
restricts to 
\[
\widetilde{2}\colon \Sigma^{0,1}\bbone_{\bbF_2} \to \bbone_{\bbF_2}
\]
on the first summand. As $\bbZ_{(2)}\otimes \pi_{\ast,\ast}\bbone_{\bbF_2}$ is $\widetilde{2}$-torsion in positive stems, the synthetic transfer therefore remains a split surjection on $2$-primary components after restriction to the summand $\nu_{\bbF_2}(\Sigma^\infty BC_2)$.
\end{proof}

This concludes our proof of the $\bbF_2$-synthetic Kahn--Priddy theorem. One basic consequence is the existence of a map of spectral sequences converging to the stable Hopf invariant: 
\begin{center}\begin{tikzcd}
\Ext^{s,f}_{\cl}(\bbS)\ar[r,"H_{\alg}"]\ar[d,Rightarrow]&\Ext_\cl^{s,f-1}(\Sigma^\infty_+ BC_2)\ar[d,Rightarrow]\\
\pi_s\bbS\ar[r,"H"]&\pi_s\Sigma^\infty_+ BC_2
\end{tikzcd}.\end{center}
Here, $\Ext^{s,f}_{\cl}(X) = \Ext_{\calA^\vee}^f(\Sigma^{s+f}\bbF_2,H_\ast(X;\bbF_2))$ is graded by stem of convergence and filtration. Both the source and target of $H_{\alg}$ may be efficiently computed via minimal resolutions, and it would be highly desirable to have a homological description of $H_{\alg}$ which is amenable to algorithmic computation in these terms. Although we do not quite have this, we can at least give an algebraic description of $H_{\alg}$.

Write $\bbF_2^\bbR \in \SH(\bbR)$ and $\ul{\bbF}_2 \in \Sp^{C_2}$ for the objects representing mod $2$ motivic cohomology and Bredon cohomology respectively.

\begin{lemma}\label{lem:radams}
Betti realization $\Be\colon \SH(\bbR)\to\Sp^{C_2}$ and the unit map $\Be(\bbF_2^\bbR)\simeq\ul{\bbF}_2 \to \bbF_2^h$ induce equivalences
\[
\Upsilon_\star^{\bbF_2^\bbR}(\nu(\Sigma^{2n,n}\bbone_\bbR))\simeq \Upsilon_\star^{\ul{\bbF}_2}(\nu(\Sigma^{n(1+\sigma)}\bbS_{C_2}))\simeq \Upsilon_\star^{\bbF_2^h}(\nu(\Sigma^{n(1+\sigma)}\bbS_{C_2}))
\]
for $n \geq 0$.
\end{lemma}
\begin{proof}
The equivalence $\Upsilon_\star^{\bbF_2^\bbR}(\nu(\Sigma^{2n,n}\bbone_\bbR))\simeq \Upsilon_\star^{\ul{\bbF}_2}(\nu(\Sigma^{n(1+\sigma)}\bbS_{C_2}))$ for $n\geq 0$ is induced by an equivalence
\[
F(\bbone_\bbR,\Sigma^{2n,n}(\bbF_2^\bbR)^{\otimes\bullet+1})\simeq (\Sigma^{n(1+\sigma)}\ul{\bbF}_2^{\otimes\bullet+1})^{C_2}
\]
of cosimplicial spectra, see \cite{duggerisaksen2017z2,belmontguillouisaksen2021c2}. In particular, $a_\sigma = \Be(\rho)$ acts injectively on the cobar complex $\pi_\star^{C_2}\ul{\bbF}_2^{\bullet+1}$ in nonpositive weights, and so \cref{lem:fibf2} may be carried out with $\ul{\bbF}_2$ in place of $\bbF_2^h$ provided $n \geq 0$; note that the Adams spectral sequence based on $\Phi^{C_2}\ul{\bbF}_2$ is again isomorphic to the classical Adams spectral sequence. The unit map $\ul{\bbF}_2 \to \bbF_2^h$ thus induces a map of fiber sequences
\begin{center}\begin{tikzcd}
\Bock_\star^{\bbF_2}(\Sigma^{0,1}\nu(\Sigma^n P^\infty_n)) \ar[r]\ar[d,"\simeq"]& \Bock_\star^{\bbF_2^h}(\nu(\Sigma^{n(1+\sigma)}\bbS_{C_2})) \ar[r,"\Phi^{C_2}"]\ar[d]& \Bock_\star^{\bbF_2}(\Sigma^n\bbS)\ar[d,"\simeq"]\\
\Bock_\star^{\bbF_2}(\Sigma^{0,1}\nu(\Sigma^n P^\infty_n)) \ar[r]& \Bock_\star^{\ul{\bbF}_2}(\nu(\Sigma^{n(1+\sigma)}\bbS_{C_2})) \ar[r,"\Phi^{C_2}"]& \Bock_\star^{\bbF_2}(\Sigma^n\bbS)
\end{tikzcd}.\end{center}
As the left and right vertical maps are equivalences, so is the middle.
\end{proof}

Write
\[
\Ext_\bbR^{(s,w),f} \Rightarrow\pi_{s,w}(\bbone_\bbR)_{\bbF_2^\bbR}^\wedge
\]
for the $\bbR$-motivic Adams spectral sequence, graded by degree of convergence and filtration.

\begin{cor}
There is a split short exact sequence of spectral sequences
\[
0 \to \Ext_\cl^{s-w,f-1}(P^\infty_w) \to \Ext_\bbR^{(s,w),f}\to \Ext_\cl^{s-w,f}(\bbS) \to 0
\]
for $w\geq 0$, realizing $\Ext_\cl^{s-w,f-1}(P^\infty_w)$ as the $\rho$-torsion subgroup of $\Ext_\bbR^{(s,w),f}$.
\end{cor}
\begin{proof}
This follows from \cref{lem:radams} in light of the fact that
\[
\Phi^{C_2}\colon \Bock_\star^{\bbF_2^h}(\nu(\Sigma^{n(1+\sigma)}\bbS_{C_2})) \to \Bock_\star^{\bbF_2}(\Sigma^n\bbS)
\]
is split by $a_\sigma^n\circ \infl_e^{C_2}$ for $n \geq 0$.
\end{proof}

Thus we may identify
\[
\Ext_\cl^{s,f-1}(\Sigma^\infty_+BC_2)\subset\Ext_\bbR^{(s,0),f}
\]
as the $\rho$-torsion subgroup. Write
\[
c\colon \Ext_\cl^{s,f} \to \Ext_\bbR^{(s,0),f},\qquad P\colon \Ext_\cl^{s,f} \to \Ext_\bbR^{(2s+f,s+f),f}
\]
for the Hurewicz map (see \cite[Section 2]{lishiwangxu2019hurewicz} and \cite[Section 5.1]{balderramaculverquigley2025motivic}) and the lift of $\Sq^0$ (see \cite[Section 3.2]{balderramaculverquigley2025motivic} and \cite[Eq.\ 4]{balderrama2024total}) respectively.

\begin{prop}\label{prop:sec}
The homomorphism
\[
H_{\alg}\colon \Ext_\cl^{s,f}(\bbS) \to \Ext_\cl^{s,f-1}(\Sigma^\infty_+ BC_2)\subset \Ext_\bbR^{(s,0),f}
\]
is given by
\[
H_{\alg}(x) = \rho^{s+f} P(x) - c(x).
\]
\end{prop}
\begin{proof}
By construction and \cref{lem:radams}, the map
\[
H_{\alg}\colon \Ext_\cl^{s,f}(\bbS) \to \Ext_\cl^{s,f-1}(\Sigma^\infty_+ BC_2)\subset \Ext_\bbR^{(s,0),f}
\]
of spectral sequences may be identified as the one induced by the map
\[
N_e^{C_2} - \infl_e^{C_2} \colon \Omega^\infty(\bbF_2^{\otimes\bullet+1}) \to \Omega^\infty(\ul{\bbF}_2^{\otimes\bullet+1})^{C_2}
\]
of cosimplicial spaces. The given formula now follows from \cite[Theorem 1.1.1]{balderrama2024total}.
\end{proof}

\begin{ex}
$H_{\alg}(h_n) = \tau^{\lfloor 2^{n-1}\rfloor}h_n$.
\tqed
\end{ex}

\subsection{The \texorpdfstring{$MU$}{MU}-synthetic Kahn--Priddy theorem}
Fix a prime $p$. Under the splitting $\Sigma^\infty_+BC_p\simeq \bbS \oplus\Sigma^\infty BC_p$, the transfer may be written as a sum of two maps
\[
p\colon \bbS \to \bbS,\qquad \widetilde{\tr}\colon \Sigma^\infty BC_p \to \bbS.
\]
As $MU_\ast\Sigma^\infty BC_p$ is concentrated in odd degrees, $\widetilde{\tr}$ vanishes in $MU$-homology for degree reasons, and so we obtain a synthetic transfer
\[
\tr_{MU}\colon \bbone_{MU}\oplus \Sigma^{0,1}\nu_{MU}(\Sigma^\infty BC_p) \to \bbone_{MU}.
\]
The $MU$-synthetic Kahn--Priddy theorem is now proved in essentially the same way as the $\bbF_2$-synthetic Kahn--Priddy theorem, using tom Dieck's homotopical cobordism $MU_{C_p}$ as a good $C_p$-equivariant lift of $MU$. However it is complicated by the above slightly different form of the synthetic transfer, and the closely related fact that the Euler class is not regular on $MU_{C_p}$. For our purposes, the following analogue of \cref{lem:fibf2} is sufficient.

\begin{lemma}\label{lem:fibmu}
The synthetic transfer sits in a diagram
\begin{center}\begin{tikzcd}
\Bock_\star^{MU}(\nu(\bbS)) \oplus \Bock_\star^{MU}(\Sigma^{0,1}\nu(\Sigma^\infty BC_p))\ar[dr,"\tr_{MU}"']\ar[r]&\Bock_\star^{MU_{C_p}}(\nu(\bbS_{C_p}))\ar[r,"\Phi^{C_p}"]\ar[d,"\res^{C_p}_e"]&\Bock_\star^{\Phi^{C_p}MU}(\nu(\bbS))\\
&\Bock_\star^{MU}(\nu(\bbS))
\end{tikzcd}\end{center}
in which the row is a fiber sequence.
\end{lemma}
\begin{proof}
Consider the fiber sequence of cosimplicial spectra
\begin{center}\begin{tikzcd}[column sep=small]
MU^{\otimes\bullet+1} \otimes\Sigma^\infty_+ BC_p\ar[r,"\tr"]&(MU_{C_p}^{\otimes\bullet+1})^{C_p}\ar[r]&(\Phi^{C_p}MU)^{\otimes\bullet+1}\ar[r]& MU^{\otimes \bullet+1}\otimes\Sigma\Sigma^\infty_+ BC_p
\end{tikzcd}\end{center}
arising from isotropy separation for $MU_{C_p}^{\otimes\bullet+1}$. As $MU_{C_p}^{\otimes\bullet+1}$ has homotopy groups concentrated in even degrees \cite{dieck1970bordism,hausmann2022global}, the map
\[
\widetilde{\tr}\colon MU^{\otimes\bullet+1}\otimes\Sigma^\infty BC_p \to (MU_{C_p}^{\otimes\bullet+1})^{C_p}
\]
vanishes on homotopy groups, and therefore induces a boundary map
\[
\Bock_\star^{MU}(\Sigma^{0,1}\nu(\Sigma^\infty BC_p)) \to \Bock_\star^{MU_{C_p}}(\nu(\bbS_{C_p}))
\]
sitting in the desired diagram. We must show that the row is a fiber sequence. We may quotient out the isotropy separation sequence by $MU^{\otimes\bullet+1}$ to obtain a fiber sequence
\begin{center}\begin{tikzcd}[column sep=small]
MU^{\otimes\bullet+1}\otimes\Sigma^\infty BC_p\ar[r]&\dfrac{(MU_{C_p}^{\otimes\bullet+1})^{C_p}}{ \tr(MU^{\otimes\bullet+1})}\ar[r]&(\Phi^{C_p}MU)^{\otimes\bullet+1}\ar[r]&MU^{\otimes\bullet+1}\otimes \Sigma \Sigma^\infty BC_p
\end{tikzcd}.\end{center}
As the transfer $\pi_\ast MU^{\otimes\bullet+1} \to \pi_\ast (MU_{C_p}^{\otimes\bullet+1})^{C_p}$ is an injection and $\pi_\ast (MU_{C_p}^{\otimes\bullet+1})^{C_p}$ is concentrated in even degrees, applying $\Tot\tau_{<\star}\tau_{\geq 0}(\bs)$ to the last three terms of this produces a fiber sequence
\begin{center}\begin{tikzcd}
\dfrac{\Bock_\star^{MU_{C_p}}(\nu(\bbS_{C_p}))}{\tr(\Bock_\star^{MU}(\nu(\bbS)))}\ar[r,"\Phi^{C_p}"]&\Bock_\star^{\Phi_{C_p}MU}(\nu(\bbS))\ar[r]&\Bock_\star^{MU}(\Sigma^{1,1}\nu(\Sigma^\infty BC_p))
\end{tikzcd}.\end{center}
Rotating this fiber sequence, we obtain a diagram
\begin{center}\begin{tikzcd}[column sep=small]
\Bock_\star^{MU}(\nu(\bbS))\ar[d]\ar[r,equals]&\Bock_\star^{MU}(\nu(\bbS))\ar[r]\ar[d]&0\ar[d]\\
\Bock_\star^{MU}(\nu(\bbS)) \oplus \Bock_\star^{MU}(\Sigma^{0,1}\nu(\Sigma^\infty BC_p))\ar[r]\ar[d]&\Bock_\star^{MU_{C_p}}(\nu(\bbS_{C_p}))\ar[r,"\Phi^{C_p}"]\ar[d]&\Bock_\star^{\Phi^{C_p}MU}(\nu((\bbS))\ar[d,equals]\\
\Bock_\star^{MU}(\Sigma^{0,1}\nu(\Sigma^\infty BC_p))\ar[r]&\dfrac{\Bock_\star^{MU_{C_p}}(\nu(\bbS_{C_p}))}{\tr(\Bock_\star^{MU}(\nu(\bbS)))}\ar[r,"\Phi^{C_p}"]&\Bock_\star^{\Phi_{C_p}MU}(\nu(\bbS))
\end{tikzcd}\end{center}
in which the columns are cofiber sequences and top and bottom rows are fiber sequences. Thus the middle row is a fiber sequence as claimed.
\end{proof}

As $MU_{C_p}$ is the underlying $C_p$-spectrum of a globally ultracommutative ring spectrum \cite{schwede2018global}, there are two naturally occuring maps $MU \to \Phi^{C_p}MU$ of ring spectra: the unit map
\[
\eta = \Phi^{C_p}\circ\infl_e^{C_p}\colon MU \to (MU_{C_p})^{C_p} \to \Phi^{C_p}MU
\]
arising from the global structure of $MU$, and the geometric norm
\[
\psi\colon MU \to \Phi^{C_p}MU = \Phi^{C_p}\left(N_e^{C_p}MU \to MU_{C_p}\right)
\]
obtained from the norm on $MU_{C_p}$ after passing to geometric fixed points. However general theory says that any two maps of ring spectra induce the same map on associated Adams-type spectral sequence, and so the diagram
\begin{center}\begin{tikzcd}[row sep=small]
&\Bock_\star^{\Phi^{C_p}MU}(\nu(Y))\ar[dd,equals]\\
\Bock_\star^{MU}(\nu(Y))\ar[ur,"\eta"]\ar[dr,"\psi"']\\
&\Bock_\star^{\Phi^{C_p}MU}(\nu(Y))
\end{tikzcd}\end{center}
commutes for any spectrum $Y$. In fact this diagram consists of equivalences, as $\Phi^{C_p}MU$ is a free $MU$-module, although we do not need this fact.

\begin{samepage}
\begin{theorem}\label{thm:musynthkp}
There is a preferred lift in the diagram
\begin{center}\begin{tikzcd}
&\Omega^\infty_{\ast,\ast}\bbone_{MU}\oplus \Sigma^{0,1}\nu(\Sigma^\infty BC_p)\ar[d,"\Omega^\infty_{\ast,\ast}\tr_{MU}"]\\
\Omega^\infty_{\ast,\ast}\bbone_{MU}\ar[r,"\varphi_p"]\ar[ur,dashed]&\Omega^\infty_{\ast,\ast}\bbone_{MU}
\end{tikzcd},\end{center}
which for $(\ast,\ast) = (0,\infty)$ agrees with that in \cref{thm:lkpg} for $A=S^0$ and $L$ the identity.
\end{theorem}
\end{samepage}
\begin{proof}
With everything in place, the proof is now identical to that of \cref{thm:newskp}, only using \cref{lem:fibmu} in place of \cref{lem:fibf2} and the above discussion in place of the fact that the Tate Frobenius and unit map $\bbF_2 \to \bbF_2^t$ agree.
\end{proof}

\begin{cor}\label{cor:musynthkp}
The transfer induces compatibly split surjections
\begin{align*}
\bbZ_{(p)}\otimes \pi_{s,w}\nu_{MU}(\Sigma^\infty B\Sigma_p) &\to \bbZ_{(p)}\otimes \pi_{s,w-1}\bbone_{MU},\\
\pi_{s,w}\nu_{MU}(\Sigma^\infty B\Sigma_p) \otimes C(\tau^m) &\to \pi_{s,w-1}C(\tau^m)
\end{align*}
for all $m,s,w\geq 1$.
\end{cor}
\begin{proof}
With $\Sigma^\infty_+ BC_p$ in place of $\Sigma^\infty B\Sigma_p$, this follows from \cref{thm:musynthkp}. The further reduction to $\Sigma^\infty B\Sigma_p$ after $p$-localization follows just as in the proof of \cref{prop:sigmap}.
\end{proof}

\begingroup
\raggedright
\bibliography{refs}
\bibliographystyle{alpha}
\endgroup

\end{document}